\documentclass[12pt]{article}
\usepackage{lmodern}
\usepackage{microtype}

\usepackage[english]{babel}
\usepackage[all]{xy}
\usepackage{fancyhdr}
\usepackage{setspace, amsmath, amsthm, amssymb, amsfonts, amscd, epic, graphicx, ulem, dsfont}
\usepackage{amsthm}
\usepackage[T1]{fontenc}
\usepackage{cases,color}
\newtheorem{thm}{Theorem}

\newtheorem{proposition}[thm]{Proposition}
\newtheorem{corollary}[thm]{Corollary}

\makeatletter

\@addtoreset{equation}{section}
\makeatother

\title{Ricci Solitons on the Lie Group $Sol^3_{m,n}$ and Applications}

\author{Mohammed El Amine Mekki\footnote{Ecole nationale Sup\'{e}rieure d'Informatique, Algeria. Email: m\_mekki@esi.dz}, Ahmed Mohammed Cherif\footnote{University Mustapha Stambouli Mascara, Mascara 29000, Algeria. Email: a.mohammedcherif@univ-mascara.dz}}
\date{}

\begin{document}
\maketitle
	
\begin{abstract}
In this paper, we study Ricci solitons on the three-dimensional solvable Lie group $\mathrm{Sol}^{3}_{m,n}$ equipped with a left-invariant Riemannian metric, viewed as a generalization of the classical $\mathrm{Sol}^{3}$ geometry. We investigate harmonic maps, harmonic sections, and geodesic curves, including the geodesic properties of the integral curves of the Ricci soliton vector field. We also characterize harmonic linear maps from $\mathrm{Sol}^{3}_{m,n}$ into Euclidean spaces.
\begin{flushleft}
\end{flushleft}
\end{abstract}

%%%%%%%%%%%%%%%%%%%%%%%%%%%%%%%%%%%%%%%%%%%%%%%%%%%%%%%%%%%%%%%%%%%%%%%%%%%%%%%%%%%%%%%%%%%%%%%%%%%%%%%%%%%%%%%%%%%%%%%%%%%%%%%%%%%%%

\section{Introduction}

Let $(M,g)$ be a Riemannian manifold. We denote by $\nabla$, $R$, and
$\operatorname{Ric}$ the Levi--Civita connection, the Riemann curvature
tensor, and the Ricci curvature tensor of $(M,g)$, respectively. These are
defined by (see \cite{ON})
\begin{align}
R(X,Y)Z
&=
[\nabla_X,\nabla_Y]Z-\nabla_{[X,Y]}Z,\label{eq1.2}\\
\operatorname{Ric}(X,Y)
&=
g\bigl(R(X,e_i)e_i,Y\bigr),\nonumber
\end{align}
where $\{e_i\}$ is a local orthonormal frame on $(M,g)$ and
$X,Y,Z\in\Gamma(TM)$.\\
A Ricci soliton on a Riemannian manifold $(M,g)$ is a quadruple
$(M,g,\xi,\lambda)$ consisting of a smooth vector field $\xi$ and a real
constant $\lambda$ satisfying
\begin{equation}\label{eq1.6}
\operatorname{Ric}
+\frac{1}{2}\mathcal{L}_{\xi}g
=
\lambda g,
\end{equation}
where $\mathcal{L}_{\xi}g$ denotes the Lie derivative of the metric $g$ along the vector field $\xi$.
According to the sign of the constant $\lambda$, a Ricci soliton is said to be
\begin{itemize}
    \item shrinking if $\lambda>0$,
    \item steady if $\lambda=0$,
    \item expanding if $\lambda<0$.
\end{itemize}
If the vector field $\xi$ is the gradient of a smooth function
$f\in C^\infty(M)$, that is, $\xi=\operatorname{grad}f$,
then $(M,g,\nabla f,\lambda)$ is called a gradient Ricci soliton,
and the function $f$ is referred to as its potential function. In
this case, equation \eqref{eq1.6} is equivalent to
\begin{equation}\label{eq1.7}
\operatorname{Ric}
+\operatorname{Hess}f
=
\lambda g,
\end{equation}
where $\operatorname{Hess}f$ denotes the Hessian of the smooth function $f$.
If the vector field $\xi$ is either identically zero or a Killing vector field, that is,
$\mathcal{L}_{\xi}g=0,$
then equation \eqref{eq1.6} reduces to
$\operatorname{Ric}=\lambda g,$
which is precisely the Einstein equation. In this case, $(M,g)$ is an Einstein manifold with Einstein constant $\lambda$. A Ricci soliton is said to be non-trivial if the underlying Riemannian metric $(M,g)$ is not Einstein. Ricci solitons play a fundamental role in differential geometry, particularly in the study of the Ricci flow. They arise naturally as self-similar solutions to the Ricci flow and frequently appear as singularity models. Consequently, Ricci solitons provide valuable insights into the geometric and topological properties of Riemannian manifolds (see \cite{H1,H2}).\\
The three-dimensional solvable Lie group $\mathrm{Sol}^{3}$, one of
Thurston's eight model geometries, can be generalized by considering the
semidirect product $\mathbb{R}^{2}\rtimes_{U}\mathbb{R}$, which defines the
family of solvable Lie groups $\mathrm{Sol}^{3}_{m,n}$ (see
\cite{Heber,Tamaru,Scott}). Here,
\begin{equation}\label{eq1.10}
L=
\begin{pmatrix}
m&0\\
0&-n
\end{pmatrix},
\quad
U(t)=\exp(tL)=
\begin{pmatrix}
{\rm e}^{mt}&0\\
0&{\rm e}^{-nt}
\end{pmatrix},
\end{equation}
where $m,n>0$. The classical Sol geometry is recovered for
$m=n=1$, namely $\mathrm{Sol}^{3}_{1,1}=\mathrm{Sol}^{3}$.
In this work, we consider the left-invariant Riemannian metric
\[
g={\rm e}^{-2mt}\,dx^{2}+{\rm e}^{2nt}\,dy^{2}+dt^{2},
\]
which is a two-parameter perturbation of the canonical metric of
$\mathrm{Sol}^{3}$. A left-invariant frame on $\mathrm{Sol}^{3}_{m,n}$ is
given by
\begin{equation}\label{eq1.13}
e_{1}={\rm e}^{mt}\frac{\partial}{\partial x},
\quad
e_{2}={\rm e}^{-nt}\frac{\partial}{\partial y},
\quad
e_{3}=\frac{\partial}{\partial t}.
\end{equation}
A smooth map $\varphi:(M,g)\longrightarrow(N,h)$ between two Riemannian
manifolds has the energy on a compact domain $D\subset M$ defined by
\begin{equation}\label{eq1.1*}
E(\varphi;D)=\frac{1}{2}\int_D|d\varphi|^2\,v^g,
\end{equation}
where $v^g$ is the volume element of $(M,g)$ and $|d\varphi|$ denotes the
Hilbert--Schmidt norm of the differential $d\varphi$.
A map $\varphi$ is called harmonic if it is a critical point of the
energy functional \eqref{eq1.1*}. Its Euler--Lagrange equation is
\begin{equation}\label{eq1.2*}
\tau(\varphi)=\operatorname{Tr}_g\nabla d\varphi
=
\nabla^\varphi_{e_i}d\varphi(e_i)
-d\varphi(\nabla^M_{e_i}e_i)
=0,
\end{equation}
where $\{e_i\}$ is a local orthonormal frame on $(M,g)$,
$\nabla^{M}$ is the Levi--Civita connection of $(M,g)$, and
$\nabla^\varphi$ is the pull-back connection on $\varphi^{-1}TN$
(see \cite{BW,ES}).\\
A smooth vector field $X$ on $(M,g)$ is called a
harmonic section if it is a critical point of the vertical energy
\begin{equation}\label{eq1.3}
{E}^v(X;D)
=
\frac12
\int_D|\nabla X|^2\,v^g.
\end{equation}
Its Euler--Lagrange equation is
\[
\overline{\Delta}X
=
\operatorname{Tr}_g\nabla^2X
=
0.
\]
In \cite{Onda}, Onda proved that the three-dimensional Heisenberg group equipped with a left-invariant Lorentzian metric admits a non-gradient shrinking Lorentz Ricci soliton. More recently, Li, Cherif, and Xie \cite{LCX} investigated Ricci solitons, harmonic maps, and harmonic vector fields on the four-dimensional nilpotent Lie group $Nil^4$. Inspired by the methods developed in \cite{LCX}, we extend this approach to the solvable Lie group $\mathrm{Sol}^3_{m,n}$.\\
Motivated by these works, we investigate Ricci soliton structures on the three-dimensional solvable Lie group $\mathrm{Sol}^3_{m,n}$ endowed with its canonical left-invariant Riemannian metric. We determine the existence of left-invariant Ricci solitons and study their geometric properties. We also derive the geodesic equations of $(\mathrm{Sol}^3_{m,n},g)$ and investigate its geodesic curves. In particular, we examine the integral curves of the Ricci soliton vector field and determine the necessary and sufficient conditions under which they are geodesics.\\
Liouville-type theorems for harmonic maps have attracted considerable attention in differential geometry. In \cite{cherif}, it was proved that every harmonic map from a compact orientable Riemannian manifold without boundary into a non-trivial Ricci soliton $(N,h,\xi,\lambda)$ satisfying either
$\operatorname{Ric}>\lambda h$ or $\operatorname{Ric}<\lambda h$ must be constant. As an application, we establish a Liouville-type theorem for harmonic maps with values in $(\mathrm{Sol}^3_{m,n},g)$. We further study linear maps from $(\mathrm{Sol}^3_{m,n},g)$ into the Euclidean space $\mathbb{R}^d$ and characterize those that are harmonic. Finally, we derive necessary and sufficient conditions for vector fields whose components depend only on the variable $t$ on $(\mathrm{Sol}^3_{m,n},g)$ to define harmonic sections.

\section{Ricci Solitons on $(\mathrm{Sol}^{3}_{m,n},g)$}

\begin{thm}\label{th1}
Let $(\mathrm{Sol}^{3}_{m,n},g)$ be the three-dimensional solvable Lie group
equipped with its canonical left-invariant Riemannian metric. A vector field
$\xi$ defines a Ricci soliton if and only if
$$\xi=
{\rm e}^{-mt}\left[\left(m c_1-mn-n^2\right)x+c_2\right]e_1
-{\rm e}^{nt}\left[\left(n c_1+mn+m^2\right)y+c_3\right]e_2
+c_1e_3 ,$$
where $c_1,c_2,c_3\in\mathbb{R}$ are constants.
Moreover, the Ricci soliton
$(\mathrm{Sol}^{3}_{m,n},g,\xi,\lambda)$
is expanding, with
$
\lambda=-m^2-n^2 .
$
\end{thm}

\begin{proof}
Let
$
\xi=\alpha_1e_1+\alpha_2e_2+\alpha_3e_3
$
be a smooth vector field on $(\mathrm{Sol}^{3}_{m,n},g)$, where
$\alpha_i\in C^\infty(\mathrm{Sol}^{3}_{m,n})$, $i=1,2,3$.
The vector field $\xi$ defines a Ricci soliton if and only if
\begin{equation}\label{eq2.1}
\operatorname{Ric}_{ij}
+\frac12
\left(
g(\nabla_{e_i}\xi,e_j)
+
g(\nabla_{e_j}\xi,e_i)
\right)
=
\lambda\delta_{ij},
\quad \forall i,j=1,2,3,
\end{equation}
for some constant $\lambda$.
Consider the left-invariant orthonormal frame
\[
e_1={\rm e}^{mt}\frac{\partial}{\partial x},
\quad
e_2={\rm e}^{-nt}\frac{\partial}{\partial y},
\quad
e_3=\frac{\partial}{\partial t}.
\]
The non-zero covariant derivatives of the Levi--Civita connection are
\[
\nabla_{e_1}e_1=me_3,\quad
\nabla_{e_1}e_3=-me_1,
\]
\[
\nabla_{e_2}e_2=-ne_3,\quad
\nabla_{e_2}e_3=ne_2 .
\]
Moreover, the Ricci tensor is diagonal and given by
\[
\bigl(\operatorname{Ric}_{ij}\bigr)=
\begin{pmatrix}
-m(m-n) & 0 & 0\\
0 & n(m-n) & 0\\
0 & 0 & -(m^2+n^2)
\end{pmatrix}.
\]
Put
$
\beta_{ij}=g(\nabla_{e_i}\xi,e_j).
$
Using the above connection coefficients, we obtain
\[
(\beta_{ij})=
\begin{pmatrix}
e_1(\alpha_1)-m\alpha_3&
e_1(\alpha_2)&
e_1(\alpha_3)+m\alpha_1\\
e_2(\alpha_1)&
e_2(\alpha_2)+n\alpha_3&
e_2(\alpha_3)-n\alpha_2\\
e_3(\alpha_1)&
e_3(\alpha_2)&
e_3(\alpha_3)
\end{pmatrix}.
\]
Substituting this expression and the Ricci tensor into equation
\eqref{eq2.1}, we obtain the following system
\[
\begin{cases}
e_1(\alpha_1)-m\alpha_3-m(m-n)=\lambda,\\[2mm]
e_2(\alpha_2)+n\alpha_3+n(m-n)=\lambda,\\[2mm]
e_3(\alpha_3)-(m^2+n^2)=\lambda,\\[2mm]
e_1(\alpha_2)+e_2(\alpha_1)=0,\\[2mm]
e_1(\alpha_3)+e_3(\alpha_1)+m\alpha_1=0,\\[2mm]
e_2(\alpha_3)+e_3(\alpha_2)-n\alpha_2=0.
\end{cases}
\]
From the third equation of the system, $e_3(\alpha_3)-(m^2+n^2)=\lambda,$
and since $e_3=\partial_t,$ we obtain $\partial_t\alpha_3=\lambda+m^2+n^2.$
Integrating with respect to $t$ yields
\[
\alpha_3=(\lambda+m^2+n^2)t+H(x,y),
\]
where $H$ is independent of $t$. Next, the second equation of the system gives
$e^{-nt}\partial_y\alpha_2+n\alpha_3+n(m-n)=\lambda,$ or equivalently,
$\partial_y\alpha_2=e^{nt}(\lambda-n\alpha_3-n(m-n)).$
Differentiating with respect to $t$, we obtain
\[\partial_t\partial_y\alpha_2=ne^{nt}(\lambda-n\alpha_3-n(m-n))-ne^{nt}\partial_t\alpha_3.\]
Using $\partial_t\alpha_3=\lambda+m^2+n^2,$ it follows that
\[\partial_t\partial_y\alpha_2=ne^{nt}\left(\lambda-n\alpha_3-n(m-n)-(\lambda+m^2+n^2)\right).\]
Now differentiate the sixth equation of the system,
$e_2(\alpha_3)+e_3(\alpha_2)-n\alpha_2=0,$ with respect to $y$,
$e^{-nt}\partial_y^2\alpha_3+\partial_t\partial_y\alpha_2-n\partial_y\alpha_2=0.$
Since $\alpha_3=(\lambda+m^2+n^2)t+H(x,y),$ we have $\partial_y^2\alpha_3=H_{yy}.$
Substituting the above expressions for $\partial_y\alpha_2$
and $\partial_t\partial_y\alpha_2$ gives
\[
\begin{aligned}
e^{-nt}H_{yy}
&+ne^{nt}\left(\lambda-n\alpha_3-n(m-n)-(\lambda+m^2+n^2)\right)  \\
&-ne^{nt}\left(\lambda-n\alpha_3-n(m-n)\right)=0.
\end{aligned}
\]
After simplification, all terms involving $\alpha_3$ cancel, and we obtain
\[e^{-nt}H_{yy}=ne^{nt}(\lambda+m^2+n^2).\]
Since the left-hand side is proportional to $e^{-nt}$ while the right-hand side
is proportional to $e^{nt}$, the identity holds for every $t$ only if
$\lambda=-m^2-n^2.$ Consequently, $H_{yy}=0.$
Applying the same argument to the fifth equation of the system,
$e_1(\alpha_3)+e_3(\alpha_1)+m\alpha_1=0,$
one similarly obtains $H_{xx}=0.$ Therefore,
$H(x,y)=Axy+Bx+Cy+D,$ where $A,B,C,D$ are constants.
Next, from the first equation of the system,
$e^{mt}\partial_x\alpha_1=\lambda+m\alpha_3+m(m-n)$,
we obtain $\partial_x\alpha_1=e^{-mt}\left(\lambda+m(m-n)+mH\right).$
Integrating with respect to $x$ gives
\[\alpha_1=e^{-mt}\left[(\lambda+m(m-n))x+m\int H(x,y)\,dx\right]+C_6(y,t).\]
Since $H=Axy+Bx+Cy+D,$ we have
$\int H\,dx=\frac{A}{2}x^2y+\frac{B}{2}x^2+Cxy+Dx.$
Therefore,
$$\alpha_1=e^{-mt}\left[(\lambda+m(m-n)+mD)x+mCxy+\frac{mB}{2}x^2+\frac{mA}{2}x^2y\right]+C_6(y,t).$$
Similarly, the second equation gives
$\partial_y\alpha_2=e^{nt}(\lambda-nH-n(m-n)).$
Integrating with respect to $y$,
\[\alpha_2=e^{nt}\left[(\lambda-n(m-n))y-n\int H(x,y)\,dy\right]+C_5(x,t),\]
and $\int H\,dy=\frac{A}{2}xy^2+Bxy+\frac{C}{2}y^2+Dy.$ Hence
\[\alpha_2=e^{nt}\left[(\lambda-n(m-n)-nD)y-nBxy-\frac{nC}{2}y^2-\frac{nA}{2}xy^2\right]+C_5(x,t).\]
Substituting the above expressions into the fourth equation,
$e_1(\alpha_2)+e_2(\alpha_1)=0,$ yields
$A=B=C=0.$ Therefore, $H(x,y)=D,$ where $D$ is constant.
Substituting this expression into the fifth equation of the system gives
$C_6'(t)=-mC_6(t),$ whose solution is $C_6(t)=c_6e^{-mt},$
where $c_6$ is a constant. Substituting into the sixth equation yields
$C_5'(t)=nC_5(t),$ so that $C_5(t)=c_5e^{nt},$ where $c_5$ is constant.
Consequently,
\[
\begin{aligned}
\alpha_1&=e^{-mt}\left[(\lambda+m(m-n)+mD)x+c_6\right],\\
\alpha_2&=e^{nt}\left[(\lambda-n(m-n)-nD)y+c_5\right],\\
\alpha_3&=D.
\end{aligned}
\]
\end{proof}

%%%%%%%%%%%%%%%%%%%%%%%%%%%%%%%%%%%%%%%%%%%%%%%%%%%%%%%%%%%%%%%%%%%%%%%%%%%%%%%%%%%%%%%%%%%%%%%%%%%%%%%%%%%%%%

\begin{proposition}
The Ricci soliton vector field $\xi$ is non-gradient.
\end{proposition}

\begin{proof}
Assume, by contradiction, that $\xi$ is a gradient vector field. Then there
exists a smooth function $f$ such that
$
\xi=\mathrm{grad}f .
$
With respect to the coordinate frame, the gradient of $f$ is given by
\[
\mathrm{grad}f=g^{ij}f_i\frac{\partial}{\partial x_j},
\]
where $f_i=\frac{\partial f}{\partial x_i}$, and $(g^{ij})$ is the inverse
of the metric matrix given by
\[
(g^{ij})=
\begin{pmatrix}
{\rm e}^{2mt}&0&0\\
0&{\rm e}^{-2nt}&0\\
0&0&1
\end{pmatrix}.
\]
Therefore, the gradient of $f$ is given by
\[
\mathrm{grad}f
=
{\rm e}^{2mt}f_x\frac{\partial}{\partial x}
+
{\rm e}^{-2nt}f_y\frac{\partial}{\partial y}
+
f_t\frac{\partial}{\partial t}.
\]
Note that, the vector field $\xi$ can be written in the coordinate frame as
\[
\xi=
\left[\left(mc_1-mn-n^2\right)x+c_2\right]\frac{\partial}{\partial x}
-
\left[\left(nc_1+mn+m^2\right)y+c_3\right]\frac{\partial}{\partial y}
+c_1\frac{\partial}{\partial t}.
\]
Comparing the components of $\xi$ and $\mathrm{grad}f$, we obtain
\begin{align}
{\rm e}^{2mt}f_x
&=
\left(mc_1-mn-n^2\right)x+c_2, \label{E1}\\
{\rm e}^{-2nt}f_y
&=
-\left[\left(nc_1+mn+m^2\right)y+c_3\right], \label{E2}\\
f_t&=c_1. \label{E3}
\end{align}
From \eqref{E3}, we have
$
f=c_1t+h(x,y),
$
where $h$ is independent of $t$. Hence, from \eqref{E1},
$
h_x=
{\rm e}^{-2mt}\left[\left(mc_1-mn-n^2\right)x+c_2\right].
$
Since $h_x$ depends only on $x$ and $y$, the right-hand side must be
independent of $t$. Differentiating with respect to $t$, we obtain
\[
0=
-2m {\rm e}^{-2mt}
\left[\left(mc_1-mn-n^2\right)x+c_2\right].
\]
This equality must hold for arbitrary values of $x$, which implies
$mc_1-mn-n^2=0,$ and $ c_2=0.$
Hence,
$
c_1=\frac{mn+n^2}{m}.
$
On the other hand, equation \eqref{E2} gives
$
h_y=
-{\rm e}^{2nt}
\left[\left(nc_1+mn+m^2\right)y+c_3\right].
$
Since $h_y$ is also independent of $t$, differentiating with respect to
$t$ yields
\[
0=
-2n {\rm e}^{2nt}
\left[\left(nc_1+mn+m^2\right)y+c_3\right].
\]
We obtain
$nc_1+mn+m^2=0$ and $c_3=0,$
and consequently
$
c_1=-\frac{mn+m^2}{n}.
$
Therefore, the existence of a smooth function $f$ satisfying
$\xi=\mathrm{grad}f$ would require simultaneously
\[
\frac{mn+n^2}{m}
=
-\frac{mn+m^2}{n}.
\]
Multiplying by $mn$ gives
$
n(mn+n^2)=-m(mn+m^2),
$
or equivalently,
$
n^2(m+n)=-m^2(m+n),
$
which implies
$
(m+n)(m^2+n^2)=0.
$
This is impossible. Hence no
such smooth function $f$ exists, and the Ricci soliton vector field $\xi$
is non-gradient.
\end{proof}

%%%%%%%%%%%%%%%%%%%%%%%%%%%%%%%%%%%%%%%%%%%%%%%%%%%%%%%%%%%%%%%%%%%%%%%%%%%%%%%%%%%%%%%%%%%%%%%%%%%%%%%%%%%%%%%%

\section{Applications to Harmonic Maps}

\begin{proposition}
The components
\[
\xi_1=\left(mc_1-mn-n^2\right)x+c_2,\quad
\xi_2=-\left(nc_1+mn+m^2\right)y-c_3,\quad
\xi_3=c_1,
\]
of the Ricci soliton vector field $\xi$ are harmonic functions on $(Sol^3_{m,n},g)$.
\end{proposition}

\begin{proof}
The result follows directly from the definition of the Laplace operator.
For any smooth function $f$ on $(Sol^3_{m,n},g)$, we have
\[
\Delta(f)
=
e_i(e_i(f))
-(\nabla_{e_i}e_i)(f).
\]
Applying this formula to each component $\xi_j$, $j=1,2,3$, and using
the expressions of the Levi-Civita connection, we obtain
$
\Delta(\xi_1)=\Delta(\xi_2)=\Delta(\xi_3)=0.
$
Therefore, each component of the Ricci soliton vector field $\xi$ is a
harmonic function on $(Sol^3_{m,n},g)$.
\end{proof}

\begin{thm}\label{th2}
Every harmonic map from a compact orientable Riemannian manifold without
boundary into the Riemannian manifold $(Sol^3_{m,n},g)$ is constant.
\end{thm}

\begin{proof}
Let $V=V_1e_1+V_2e_2+V_3e_3$ be an arbitrary vector field on
$Sol^3_{m,n}$. Using the value of the soliton constant
$
\lambda=-m^2-n^2,
$
we obtain
\begin{equation}\label{eq2.7}
\operatorname{Ric}(V,V)-\lambda g(V,V)
=
n(m+n)V_1^2+m(m+n)V_2^2\geq0 .
\end{equation}
Hence, the target manifold $(Sol^3_{m,n},g)$ satisfies the required
curvature condition. The Theorem \ref{th2} follows from \eqref{eq2.7} and Proposition $9$ in \cite{cherif}.
\end{proof}

\begin{thm}
Let
$
\varphi:(Sol^3_{m,n},g)\longrightarrow \mathbb{R}^{d}
$
be the smooth map defined by
\[
\varphi(x,y,t)=
\begin{pmatrix}
a_{11}x+a_{12}y+a_{13}t\\
a_{21}x+a_{22}y+a_{23}t\\
\vdots\\
a_{d1}x+a_{d2}y+a_{d3}t
\end{pmatrix},
\]
where $a_{ij}\in\mathbb{R}$. Then $\varphi$ is a harmonic map if and only if
$
m=n.
$
\end{thm}

\begin{proof}
We have
$
\varphi=(\varphi^1,\varphi^2,\ldots,\varphi^d):
Sol^3_{m,n}\longrightarrow \mathbb{R}^{d},
$
where
\[
\varphi^\alpha(x,y,t)
=
a_{\alpha 1}x+a_{\alpha 2}y+a_{\alpha 3}t,
\quad \alpha=\overline{1,d}.
\]
Since the target manifold $\mathbb{R}^{d}$ is Euclidean, the tension field
of $\varphi$ is given by
\[
\tau(\varphi)
=
\left(\Delta\varphi^1,\Delta\varphi^2,\ldots,\Delta\varphi^d\right).
\]
Therefore, $\varphi$ is harmonic if and only if
$
\Delta\varphi^\alpha=0,
$
for $\alpha=\overline{1,d}$. The metric of $Sol^3_{m,n}$ is
$
g={\rm e}^{-2mt}dx^2+{\rm e}^{2nt}dy^2+dt^2,
$
and an orthonormal frame is given by
\[
e_1={\rm e}^{mt}\frac{\partial}{\partial x},
\quad
e_2={\rm e}^{-nt}\frac{\partial}{\partial y},
\quad
e_3=\frac{\partial}{\partial t}.
\]
For a function $f$ on $Sol^3_{m,n}$, the Laplacian is
$\Delta f
=
e_i(e_i(f))-(\nabla_{e_i}e_i)(f).
$
Using the Levi-Civita connection of $(Sol^3_{m,n},g)$, we have
\[
\nabla_{e_1}e_1=me_3,\quad
\nabla_{e_2}e_2=-ne_3,\quad
\nabla_{e_3}e_3=0.
\]
Hence,
$\Delta f
=
e_1(e_1(f))+e_2(e_2(f))+e_3(e_3(f))
-me_3(f)+ne_3(f).
$
Applying this formula to the coordinate functions, we obtain
$\Delta x=-m{\rm e}^{2mt}\frac{\partial t}{\partial x}+0=0,$
$\Delta y=0,$ and $\Delta t=n-m$.
Consequently, for each component of $\varphi$,
\[
\Delta\varphi^\alpha
=
a_{\alpha1}\Delta x
+a_{\alpha2}\Delta y
+a_{\alpha3}\Delta t
=
a_{\alpha3}(n-m).
\]
Therefore, the tension field of $\varphi$ is given by
\[
\tau(\varphi)
=
(n-m)
\begin{pmatrix}
a_{13}\\
a_{23}\\
\vdots\\
a_{d3}
\end{pmatrix}.
\]
Thus, the tension field vanishes if and only if
$(n-m)a_{\alpha3}=0,$
for $\alpha=\overline{1,d}.$
For a general linear map $\varphi$ (with arbitrary coefficients
$a_{\alpha3}$), this condition is equivalent to
$m=n.$
Hence, $\varphi$ is a harmonic map if and only if
$m=n.$
\end{proof}

\section{Geodesics on $(Sol^{3}_{m,n},g)$}

\begin{thm}\label{thm-geodesic}
Let
$
\gamma:I\subseteq\mathbb{R}\longrightarrow (Sol^{3}_{m,n},g),
$
$
\gamma(s)=(x(s),y(s),t(s)),
$
be a curve. Then, $\gamma$ is a geodesic curve if and only if its coordinate functions satisfy the following system of differential equations
\[
\begin{cases}
x''-2m\,t'x'=0,\\[2mm]
y''+2n\,t'y'=0,\\[2mm]
t''+m {\rm e}^{-2mt}(x')^2-n {\rm e}^{2nt}(y')^2=0.
\end{cases}
\]
\end{thm}

\begin{proof}
The tension field of the curve $\gamma$ is given by
\[
\tau(\gamma)=\nabla^{\gamma}_{\frac{d}{dt}}\gamma',
\]
and $\gamma$ is a geodesic if and only if $\tau(\gamma)=0$. Using the orthonormal frame
$\{e_1,e_2,e_3\}$, along $\gamma$ we get
$
\gamma'=u_1(e_1\circ\gamma)+u_2(e_2\circ\gamma)+u_3(e_3\circ\gamma),
$
where
\[
u_1={\rm e}^{-mt}x',\quad
u_2={\rm e}^{nt}y',\quad
u_3=t'.
\]
Therefore,
\[
\begin{aligned}
\tau(\gamma)
&=\nabla^{\gamma}_{\frac{d}{dt}}\left[u_1(e_1\circ\gamma)+u_2(e_2\circ\gamma)+u_3(e_3\circ\gamma)\right]\\
&=u_1'(e_1\circ\gamma)+u_2'(e_2\circ\gamma)+u_3'(e_3\circ\gamma)
+u_1\nabla_{\gamma'}e_1
+u_2\nabla_{\gamma'}e_2
+u_3\nabla_{\gamma'}e_3 .
\end{aligned}
\]
Using the non-zero covariant derivatives, we obtain
\[
\begin{aligned}
\tau(\gamma)
&=(u_1'-mu_1u_3)(e_1\circ\gamma)
+(u_2'+nu_2u_3)(e_2\circ\gamma)
 +(u_3'+mu_1^2-nu_2^2)(e_3\circ\gamma).
\end{aligned}
\]
Consequently, the condition $\tau(\gamma)=0$ is equivalent to
\[
\begin{cases}
u_1'-mu_1u_3=0,\\
u_2'+nu_2u_3=0,\\
u_3'+mu_1^2-nu_2^2=0.
\end{cases}
\]
Substituting
$u_1={\rm e}^{-mt}x',$
$u_2={\rm e}^{nt}y',$
and $u_3=t',$ we obtain the geodesic equations of $(Sol^3_{m,n},g)$.
\end{proof}

\begin{thm}\label{thm-integral}
Assume that $mc_1-mn-n^2\neq0$ and $nc_1+mn+m^2\neq0$. Then, the integral curves
$\gamma(s)=(x(s),y(s),t(s))$ of the Ricci soliton vector field $\xi$, satisfying $\gamma'(s)=\xi_{\gamma(s)}$, are given by
\[
\begin{cases}
x(s)=a_1{\rm e}^{(mc_1-mn-n^2)s}-\frac{c_2}{mc_1-mn-n^2},\\[2mm]
y(s)=a_2{\rm e}^{-(nc_1+mn+m^2)s}-\frac{c_3}{nc_1+mn+m^2},\\[2mm]
t(s)=c_1s+t_0,
\end{cases}
\]
where $a_1,a_2,t_0\in\mathbb{R}$. Moreover, all these integral curves are defined for
every $s\in\mathbb{R}$; consequently, the Ricci soliton vector field $\xi$ is complete.
\end{thm}

\begin{proof}
The vector field $\xi$ can be expressed in the coordinate frame
\[
\begin{aligned}
\xi={}&
\left[\left(mc_1-mn-n^2\right)x+c_2\right]
\frac{\partial}{\partial x}
-\left[\left(nc_1+mn+m^2\right)y+c_3\right]
\frac{\partial}{\partial y}
+c_1\frac{\partial}{\partial t}.
\end{aligned}
\]
Let $\gamma(s)$ be an integral curve of $\xi$. Then
$
\gamma'(s)=\xi_{\gamma(s)}
$
is equivalent to the following system
\[
\begin{cases}
x'=\left(mc_1-mn-n^2\right)x+c_2,\\[2mm]
y'=-\left(nc_1+mn+m^2\right)y-c_3,\\[2mm]
t'=c_1.
\end{cases}
\]
Set
$A=mc_1-mn-n^2,$ and
$B=nc_1+mn+m^2.$
The first equation becomes
$
x'-Ax=c_2.
$
Using the integrating factor ${\rm e}^{-As}$, we obtain
\[
x(s)=a_1{\rm e}^{As}-\frac{c_2}{A}.
\]
Similarly, the second equation gives
$
y'+By=-c_3.
$
Multiplying by the integrating factor ${\rm e}^{Bs}$ yields
\[
y(s)=a_2{\rm e}^{-Bs}-\frac{c_3}{B}.
\]
Finally, the third equation is immediately integrated as
\[
t(s)=c_1s+t_0.
\]
\end{proof}

According to Theorem \ref{thm-geodesic}, we deduce that.

\begin{corollary}
Let $\gamma(s)=(x(s),y(s),t(s))$ be an integral curve of the Ricci soliton vector field $\xi$ given in Theorem \ref{thm-integral}. Then $\gamma$ is a geodesic of $(Sol^3_{m,n},g)$ if and only if
\[
\begin{cases}
a_{1}\left({m}^{2}c_{1}^{2}-{m}^{2}{n}^{2}-2m{n}^{3}-{n}^{4}\right)=0,\\[2mm]
a_{2}\left({n}^{2}c_{1}^{2}-{m}^{2}{n}^{2}-2{m}^{3}n-{m}^{4}\right)=0,\\[2mm]
ma_{1}^{2} \left( mc_{1}-mn-{n}^{2} \right)^{2}{\rm e}^{-2{n}^{2}s-2mt_{0}}
=na_{2}^{2}\left( nc_{{1}}+mn+{m}^{2} \right) ^{2}{\rm e}^{-2{m}^{2} s+2nt_{0}}.
\end{cases}
\]
\end{corollary}

%%%%%%%%%%%%%%%%%%%%%%%%%%%%%%%%%%%%%%%%%%%%%%%%%%%%%%%%%%%%%%%%%%%%%%%%%%%%%%%%%%%%%%%%%%%%%%%%%%%%%%%%%%%%%%%
\section{Harmonic Vector Fields on $(Sol^{3}_{m,n},g)$}

Since the left-invariant Riemannian metric $g$ depends only on the coordinate
$t$, we consider harmonic vector fields whose components depend only on $t$.

\begin{thm}\label{th3}
Let
$
X=X_1(t)e_1+X_2(t)e_2+X_3(t)e_3
$
be a vector field on $Sol^3_{m,n}$. Then, $X$ is a harmonic section
with respect to the metric $g$ if and only if
\[
\begin{cases}
X_1''+(n-m)X_1'-m^2X_1=0,\\[2mm]
X_2''+(n-m)X_2'-n^2X_2=0,\\[2mm]
X_3''+(n-m)X_3'-(m^2+n^2)X_3=0.
\end{cases}
\]
\end{thm}

\begin{proof}
Let
$
\theta_{ij}=g(\nabla_{e_i}X,e_j),
$
for $i,j=\overline{1,3}.$
Using the Levi-Civita connection of $(Sol^3_{m,n},g)$,
we obtain
\begin{eqnarray*}
\nabla_{e_1}X&=&-mX_3e_1+mX_1e_3,\\
\nabla_{e_2}X&=&nX_3e_2-nX_2e_3,\\
\nabla_{e_3}X&=&X_{1}^{'}e_1+X_{2}^{'}e_2+X_{3}^{'}e_3,
\end{eqnarray*}
where $X_{i}^{'}=e_3(X_i),$ and $X_{i}^{''}=e_3(e_3(X_i)).$
Applying once again the covariant derivative, we get
%%%%%%%%%%%%%%%%%%%%%%%%%%%%%%%%%%%%%%%%%%%%%%%%%%%%%%%%%%%%%%%%%%%%%%%%%%%%%%%%%%%%%%
\begin{eqnarray*}
\nabla_{e_1}\nabla_{e_1}X&=&-m^2X_1e_1-m^2X_3e_3,\\
\nabla_{e_2}\nabla_{e_2}X&=&-n^2X_2e_2-n^2X_3e_3,\\
\nabla_{e_3}\nabla_{e_3}X&=&X_{1}^{''}e_1+X_{2}^{''}e_2+X_{3}^{''}e_3,
\end{eqnarray*}
and the following formulas
\[
\begin{aligned}
\nabla_{\nabla_{e_1}e_1}X
&=m\nabla_{e_3}X
=m\left(X_1'e_1+X_2'e_2+X_3'e_3\right),\\
\nabla_{\nabla_{e_2}e_2}X
&=-n\nabla_{e_3}X
=-n\left(X_1'e_1+X_2'e_2+X_3'e_3\right),\\
\nabla_{\nabla_{e_3}e_3}X
&=0.
\end{aligned}
\]
Using $\overline{\Delta}X=\nabla_{e_i}\nabla_{e_i}X-\nabla_{\nabla_{e_i}e_i}X$, the rough Laplacian of $X$ is given by
\[
\begin{aligned}
\overline{\Delta}X
&=
\left(-m^2X_1e_1-m^2X_3e_3\right)
+\left(-n^2X_2e_2-n^2X_3e_3\right)
+\left(X_1''e_1+X_2''e_2+X_3''e_3\right)\\
&\qquad
-m\left(X_1'e_1+X_2'e_2+X_3'e_3\right)
+n\left(X_1'e_1+X_2'e_2+X_3'e_3\right)\\
&=
\left(X_1''+(n-m)X_1'-m^2X_1\right)e_1
+\left(X_2''+(n-m)X_2'-n^2X_2\right)e_2\\
&\qquad
+\left(X_3''+(n-m)X_3'-(m^2+n^2)X_3\right)e_3.
\end{aligned}
\]
\end{proof}

The previous theorem characterizes harmonic sections through a system of independent linear differential equations. Solving this system using the characteristic equation method, and noting that $m,n>0$ ensures real roots, we obtain the following explicit form of harmonic vector fields on $Sol^3_{m,n}$.

\begin{thm}\label{th3-B}
Let
$
X=X_1(t)e_1+X_2(t)e_2+X_3(t)e_3
$
be a vector field on $Sol^3_{m,n}$. Then, $X$ is a harmonic section
with respect to the metric $g$ if and only if
\[
\begin{aligned}
X_1(t)&=a_1{\rm e}^{\frac{m-n+\sqrt{n^2-2mn+5m^2}}{2}\,t}+b_1{\rm e}^{\frac{m-n-\sqrt{n^2-2mn+5m^2}}{2}\,t},\\[2mm]
X_2(t)&=a_2{\rm e}^{\frac{m-n+\sqrt{m^2-2mn+5n^2}}{2}\,t}+b_2{\rm e}^{\frac{m-n-\sqrt{m^2-2mn+5n^2}}{2}\,t},\\[2mm]
X_3(t)&=a_3{\rm e}^{\frac{m-n+\sqrt{5m^2-2mn+5n^2}}{2}\,t}+b_3{\rm e}^{\frac{m-n-\sqrt{5m^2-2mn+5n^2}}{2}\,t},
\end{aligned}
\]
where $a_1,a_2,a_3,b_1,b_2,b_3\in\mathbb{R}$.
\end{thm}

\section{Conclusion}

In this paper, we studied Ricci solitons on the $3$-dimensional solvable Lie group $Sol^{3}_{m,n}$ endowed with the left-invariant Riemannian metric $g$. Using an orthonormal left-invariant frame and the corresponding Levi--Civita connection, we characterized the Ricci soliton vector fields. We also investigated harmonic maps and established a Liouville-type theorem under suitable curvature assumptions. Moreover, we derived the geodesic equations of $(Sol^{3}_{m,n},g)$, studied its geodesic curves, and determined the necessary and sufficient conditions for the integral curves of the Ricci soliton vector field to be geodesics. Finally, we examined harmonic vector fields on $Sol^{3}_{m,n}$.\\
These results contribute to the understanding of the interplay between Ricci solitons, harmonic maps, harmonic vector fields, and geodesic geometry on solvable Lie groups. Future work may include the study of the stability of these Ricci solitons under the Ricci flow and the extension of these results to other homogeneous and Thurston geometries.

%%%%%%%%%%%%%%%%%%%%%%%%%%%%%%%%%%%%%%%%%%%%%%%%%%%%%%%%%%%%%%%%%%%%%%%%%%%%%%%%%%%%%%%%%%%%%%%%%%%%%%%%%%%%%%%%%%%%%%%%%%%%%%%%%%%%

%----------------------------------------------------------------------------------------

\end{document}